\documentclass[12pt]{article}
\usepackage{fullpage}
\usepackage{amsmath,amsfonts,amssymb,amsthm,graphics,amscd,MnSymbol}
\usepackage{graphicx}
\usepackage{color}

\usepackage[
pdfauthor={},
pdftitle={Bounds for eccentricity-based parameters of graphs},
pdfstartview=XYZ,
bookmarks=true,
colorlinks=true,
linkcolor=blue,
urlcolor=blue,
citecolor=blue,
bookmarks=false,
linktocpage=true,
hyperindex=true
]{hyperref}

\begin{document}

\newtheorem {Theorem}  {Theorem}[section]
\newtheorem {Lemma}[Theorem]{Lemma}
\newtheorem {Proposition}[Theorem]{Proposition}
\newtheorem {Corollary}[Theorem]{Corollary}
\newtheorem {Problem}[Theorem]{Problem}
\newtheorem {Question}[Theorem]{Question}
\newtheorem {Claim}{Claim}
\newtheorem {Observation}[Theorem]{Observation}
\theoremstyle{definition}
\newtheorem{Definition}[Theorem]{Definition}
\newtheorem{Example}[Theorem]{Example}
\newenvironment {Proof} {\noindent {\bf Proof.}}{\quad $\square$\par\vspace{3mm}}

\def\FR#1#2{\frac{#1}{#2}}
\def\grad{gradual}
\def\NN{{\mathbb N}}
\def\bQ{{\bf Q}}
\def\hS{{\hat S}}
\def\hQ{\hat{\bf Q}}
\def\hu{{\hat u}}
\def\hP{{\hat P}}
\def\hk{{\hat k}}
\def\vareps{\varepsilon}
\def\eps{\varepsilon}
\def\VEC#1#2#3{#1_{#2},\dots,#1_{#3}}
\def\Gb{\overline{G}}
\def\Kb{\overline{K}}
\def\FL#1{\left\lfloor{#1}\right\rfloor}
\def\CL#1{\left\lceil{#1}\right\rceil}
\def\CH#1#2{\binom{#1}{#2}}
\def\esub{\subseteq}
\def\bT{{\mathbf T}}
\def\bB{{\mathbf B}}
\def\bC{{\mathbf C}}
\def\C#1{{\left|#1\right|}}
\def\diam{{\rm diam}}
\def\rad{{\rm rad}}
\long\def\skipit#1{}
\def\gjoin{\diamondplus}
\def\SE#1#2#3{\sum_{#1=#2}^{#3}}
\def\la{\langle}
\def\ra{\rangle}

\date{\today}

\title{Bounds for eccentricity-based parameters of graphs}

\author{Yunfang Tang~\thanks{
Department of Mathematics, China Jiliang University, Hangzhou 310018, China.
tangyunfang8530@163.com.  Supported by the National Natural Science Foundation
of China (Grant 11701543).},\,\,
Xuli Qi~\thanks{Department of Mathematics, Hebei University of Science and Technology, Shijiazhuang 050018, China. qixuli-1212@163.com. Supported by the National Natural Science Foundation of China (Grant 11801135).},\,\,
Douglas B. West~\thanks{
Departments of Mathematics, Zhejiang Normal University and University of
Illinois. dwest@illinois.edu.  Supported by the National Natural Science
Foundation of China (Grants 11871439, 11971439, U20A2068).}
}

\maketitle

\vspace{-2pc}
\begin{abstract}
The \emph{eccentricity} of a vertex $u$ in a graph $G$, denoted by
$\vareps_G(u)$, is the maximum distance from $u$ to other vertices in $G$.  We
study extremal problems for the average eccentricity and the first and second
Zagreb eccentricity indices, denoted by $\sigma_0(G)$, $\sigma_1(G)$, and
$\sigma_2(G)$, respectively.  These are defined by
$\sigma_0(G)=\frac{1}{\C{V(G)}}\sum_{u\in V(G)}\vareps_G(u)$,
$\sigma_1(G)=\sum_{u\in V(G)}\vareps_G^2(u)$, and
$\sigma_2(G)=\sum_{uv\in E(G)}\vareps_G(u)\vareps_G(v)$.
We study lower and upper bounds on these parameters among $n$-vertex connected
graphs with fixed diameter, chromatic number, clique number, or matching number.
Most of the bounds are sharp, with the corresponding extremal graphs
characterized.

\end{abstract}

\baselineskip 16pt

\section{Introduction }
We consider simple undirected graphs and study extremal problems for parameters
that measure dispersion of vertices.  Given a graph $G$ with vertex set $V(G)$
and edge set $E(G)$, we write $d_G(u,v)$ for the distance between vertices $u$
and $v$ in $G$.  The \emph{eccentricity} of a vertex $u$ in $G$, denoted by
$\vareps_G(u)$, is $\max_{v\in V(G)} d_G(u,v)$.  The \emph{radius} $\rad(G)$
and \emph{diameter} $\diam(G)$ of $G$ are $\min_{u\in V(G)}\vareps_G(u)$ and
$\max_{u\in V(G)}\vareps_G(u)$, respectively.  Since eccentricity is infinite
for every vertex in a disconnected graph, we consider only connected graphs.

In chemical graph theory, graph invariants are called ``topological indices''.
These numerical values reflect structural properties of the graphs associated
with molecules.  They have been studied in chemical graph theory due to their
predictive capabilities for physical and chemical properties of molecules.
They have been widely used as molecular descriptors and in QSAR/QSPR
studies~\cite{TC2}.  Much of the graph-theoretic study has been on extremal
problems for relatively sparse graphs such as trees, unicyclic, and bicyclic
graphs because the graphs of molecules tend to be relatively sparse.  We
consider such problems on denser classes of graphs.

Eccentricity-based invariants are natural to study in this context, because
high eccentricity is associated with high ``dispersion'' of the vertices.
We study three such indices, called $\sigma_0$, $\sigma_1$, and $\sigma_2$.

The \emph{average eccentricity} $\sigma_0$ is defined by
$$\sigma_0(G)=\frac{1}{\C{V(G)}}\sum_{u\in V(G)}\vareps_G(u).$$
It was introduced by Konstantinova and Skorobogatov \cite{KS} in mathematical
chemistry as a molecular descriptor.  Buckley and Harary~\cite{BH} called it
the ``eccentric mean''.

Many results about $\sigma_0$ have been obtained, typically for $n$-vertex
connected graphs.  Dankelmann et al.~\cite{DGS,DOMR} gave upper bounds on
$\sigma_0$ for $n$-vertex graphs with specified minimum degree that are
connected, $K_3$-free and connected, or $C_4$-free and connected.
Dankelmann and Osaye~\cite{DO} further refined these results by fixing also the
maximum degree.  Du and Ili{\'c}~\cite{DI,I} resolved conjectures of the system
AutoGraphiX relating $\sigma_0$ to the independence number, chromatic number,
Randi{\'c} index, and spectral radius.  Dankelmann and Mukwembi~\cite{DM}
proved sharp upper bounds on $\sigma_0$ in terms of the independence number,
chromatic number, domination number, and connected domination number.
Das et al.~\cite{DMCC} presented lower and upper bounds on $\sigma_0$ in terms
of diameter, clique number, independence number, and the first Zagreb index.
He et al.~\cite{HLT} gave sharp bounds on $\sigma_0$ for trees in terms of
the number of leaves, domination number, and vertex bipartition.
Horoldagva et al.~\cite{HBDAA} determine the graphs maximizing $\sigma_0$ among
$n$-vertex connected graphs with given girth and maximum degree.  Tang and Zhou
presented lower and upper bounds on $\sigma_0$ for trees~\cite{TZ1} and upper
bounds on $\sigma_0$ for unicyclic graphs~\cite{TZ2}.

By analogy with the first and the second Zagreb indices introduced
in~\cite{GRTW,GT}, Vuki\v{c}evi\'c and Graovac \cite{VG} and Ghorbani and
Hosseinzadeh \cite{GH}  introduced ``Zagreb eccentricity indices'' using vertex
eccentricities instead of vertex degrees.
The \emph{first Zagreb eccentricity index} $\sigma_1$ is defined by
$$\sigma_1(G)=\sum_{u\in V(G)}\vareps_G^2(u),$$
and the \emph{second Zagreb eccentricity index} $\sigma_2$ is defined by
$$\sigma_2(G)=\sum_{uv\in E(G)}\vareps_G(u)\vareps_G(v).$$

Extremal results about $\sigma_1$ and $\sigma_2$ soon followed.
Xing et al.~\cite{XZT} gave bounds on $\sigma_1$ and $\sigma_2$ for $n$-vertex
trees with fixed diameter or fixed matching number, in some cases
characterizing the extremal trees and those with the second or third most
extreme values.  For example, over $n$-vertex trees both indices are minimized
by the star and maximized by the path.  They also gave bounds for general
graphs when the number of vertices, number of edges, radius, and diameter are
all fixed, plus bounds on $\sigma_1(G)+\sigma_1(\Gb)$ and
$\sigma_2(G)+\sigma_2(\Gb)$, where $\Gb$ is the complement of $G$.  Next, Du et
al.~\cite{DZT} gave sharp bounds for connected graphs with fixed numbers of
vertices and edges and for trees with fixed number of leaves, fixed matching
number, or fixed maximum degree, determining in some cases the extremal graphs.

Das et al.~\cite{DLG} showed that the star and path are the unique extremes for
both indices not just over trees, but also over $n$-vertex bipartite graphs.
For a fixed vertex bipartition, the complete bipartite graph is the unique
minimizer.  They also gave sharp lower bounds for $\sigma_1$ among $n$-vertex
graphs with diameter $d$, and for $\sigma_2$ among graphs with diameter $d$
and $m$ edges, characterizing the extremal graphs.
Qi et al.~\cite{QD,QZL} found the first few smallest and largest values of
$\sigma_1$ and $\sigma_2$ for unicyclic graphs and studied their extreme
values on trees with fixed number of vertices, domination number, maximum
degree, and bipartition size.
Li et al.~\cite{LZ,SLH} determined the bicyclic graphs ($|E(G)|=|V(G)|+1$)
with largest and second largest values of $\sigma_2$,
and they established sharp upper and lower bounds on $\sigma_1$ and
$\sigma_2$ for all cacti with fixed number of vertices and number of cycles.
For $n$-vertex trees, unicyclic graphs, and bicyclic graphs, Tang and
Qi~\cite{TQ} determined the $q$ graphs with largest values of $\sigma_1$ and
$\sigma_2$, where $q$ varies over the different classes but in each case is
linear in $n$.

As noted earlier, most of the prior study has been on very sparse graphs.
In this paper, we consider extremal problems for $\sigma_0$, $\sigma_1$, and
$\sigma_2$ over $n$-vertex connected  graphs in terms of various graph
parameters.  In particular, upper bounds on $\sigma_2$ take us away from the
realm of sparse graphs because having many edges produces many contributions
to the sum for $\sigma_2$.  We obtain upper and lower bounds among classes of
$n$-vertex graphs, most of which are sharp.  We consider bounds in terms of
fixed diameter in Sections~$2$ and $3$, fixed chromatic number or clique
number in Section~$4$, and fixed matching number in Section~$5$.

We summarize some standard notation and terminology, for completeness.

\begin{Definition}
Let $S_n$, $P_n$, $C_n$, $K_n$, and $\Kb_n$ denote the star, path, cycle,
complete graph, and edgeless graph with $n$ vertices, respectively.
The path with vertices $\VEC v1n$ in order is denoted $\la \VEC v1n\ra$.

A \emph{dominating vertex} in an $n$-vertex graph $G$ is a vertex of degree
$n-1$.  An \emph{independent set} is a set of pairwise non-adjacent vertices.
The \emph{chromatic number} $\chi(G)$ is the least $k$ such that $V(G)$ can be
expressed as the union of $k$ disjoint independent sets, called
\emph{color classes}.  A \emph{clique} is a set of pairwise adjacent vertices.
The \emph{clique number} $\omega(G)$ is the maximum size of a clique in $G$.

A \emph{matching} is a set of pairwise disjoint edges (no two have a common
endpoint).  The \emph{matching number}, written $\alpha'(G)$, is the maximum
size of a matching in $G$.
\end{Definition}

\section{Upper bounds in terms of diameter}

In this section we consider upper bounds on $\sigma_0$, $\sigma_1$, and
$\sigma_2$ over $n$-vertex connected graphs with diameter $d$.  We begin our
discussion with an elementary observation that holds because each vertex in a
graph $G$ has eccentricity at most the diameter and at least the radius of $G$.

\begin{Observation}[\rm \cite{XZT}]\label{diambd}
If $G$ is a connected graph $G$ with $n$ vertices, $m$ edges, radius $r$,
and diameter $d$, then $nr^2\le \sigma_1(G)\le nd^2$ and
$mr^2\le \sigma_2(G)\le md^2$.  [Similarly, $r\le\sigma_0(G)\le d$.]
The equalities hold if and only if all vertices of $G$ have the same
eccentricity.
\end{Observation}

Therefore, when maximizing the eccentricity indices in terms of other
parameters we want most or all vertices to have maximum eccentricity.
Note that the only $n$-vertex graph with diameter $1$ is the complete graph.
Hence in this section we restrict to $d\ge2$.

We first describe the extremal trees.  These generally are not extremal in the
class of $n$-vertex graphs with diameter $d$, but they will be useful here and
in later sections.

\begin{Definition}
A \emph{double-broom} is  a tree obtained by adding leaf neighbors to each end
vertex of a path.  Let $\bT^{n,d}$ denote the set of double-brooms with $n$
vertices and diameter $d$.  In particular, $\bT^{n,n-1}=\{P_n\}$
and $\bT^{n,2}=\{S_n\}$.
\end{Definition}

In a graph $G$ with diameter $d$, a {\it diametric path} is a shortest path
joining vertices separated by distance $d$.
A tree with diameter $d$ contains a diametric path with $d+1$ vertices
having eccentricities $d,d-1,\ldots,d-1,d$, with one copy of $d/2$ or two
copies of $(d+1)/2$ in the middle, depending on the parity of $d$.  The
contribution $g_i$ of such a path to the parameter $\sigma_i$ is:
\begin{align*}
g_0(d)&=d+(d-1)+\cdots+(d-1)+d =\CL{\FR{3d^2}4}+d,\\
g_1(d)&=d^2+(d-1)^2+\cdots+(d-1)^2+d^2= \FR{d^2+d}{12}(7d+5)-
  \begin{cases} d/4& \mbox{if $d$ is even}\\
              0& \mbox{if $d$ is odd,} \end{cases}  \\
g_2(d)&=d(d-1)+\cdots+d(d-1)=\FR{7d^3}{12}+ \FR1{12}-
  \begin{cases} 4d& \mbox{if $d$ is even}\\
                d-6& \mbox{if $d$ is odd.} \end{cases}
\end{align*}
Note that the leading terms in $g_1$ and $g_2$ are the same, as a function of
$d$.

A tree can be grown from a diametric path.  The added vertices cannot have
eccentricity more than $d$, and when added they cannot be attached to a
neighbor with eccentricity more than $d-1$.  Thus each added vertex adds at
most $d$ or $d^2$ or $d(d-1)$ to $\sigma_0$ or $\sigma_1$ or $\sigma_2$,
respectively.  Hence the members of $\bT^{n,d}$ maximize each $\sigma_i$ over
$n$-vertex trees with diameter $d$.  The values were given in~\cite{TZ1,XZT}.
We express them a bit differently below to emphasize the role of $g_i(d)$ and
the nature of the leading order terms.

\begin{Lemma}\label{TZ1} \label{treemax}
If $T$ is an $n$-vertex tree with diameter $d$, then
\begin{align*}
\sigma_0(T)&\le \FR1n[(n-d-1)d+g_0(d)] =d-\FR1n\FL{\FR{d^2}4},\\
\sigma_1(T)&\le (n-d-1)d^2+g_1(d)=nd^2-\FR{5d}{12}(d^2-1)-
  \begin{cases} d/4& \mbox{if $d$ is even}\\
              0& \mbox{if $d$ is odd,} \end{cases}  \\
\sigma_2(T)&\le (n-d-1)d(d-1)+g_2(d)= n(d^2-d)-\FR{5d^3}{12}+d+\FR1{12}-
  \begin{cases} 4d& \mbox{if $d$ is even}\\
                d-6& \mbox{if $d$ is odd.} \end{cases}
\end{align*}
In each inequality, equality holds if and only if $T\in \bT^{n,d}$.
For $i\in\{0,1,2\}$, let $f_i(n,d)$ denote the bound given above for
$\sigma_i(T)$.  For $2\le d\le n-1$, each bound is strictly increasing in $d$.
\end{Lemma}

It is tempting to try to reduce the maximization of $\sigma_0$ and $\sigma_1$
over $n$-vertex graphs to the problem for trees by using the following
observation.
\begin{Observation}\label{l0}
If $G$ is an $n$-vertex graph, then $\vareps_G(u)\le\vareps_{G-e}(u)$ for
$e\in E(G)$ and  $u\in V(G)$, and thus $\sigma_0(G)\le\sigma_0(G-e)$ and
$\sigma_1(G)\le\sigma_1(G-e)$.
\end{Observation}

Deleting an edge cannot reduce eccentricities, but unfortunately it can
increase the diameter, taking the graph out of the relevant domain.
Indeed, when $d\le n/2$, we show that trees generally do not achieve the maximum
for $\sigma_0$ or $\sigma_1$.

\begin{Example}\label{alldiam}
When $d=\FL{n/2}$, the cycle $C_n$ has diameter $d$, and all vertices have
eccentricity $d$.  Hence $\sigma_0 = d$ and $\sigma_1 = nd^2$, clearly
exceeding the values for trees listed in Lemma~\ref{treemax}.  Indeed, the
values $d$ and $nd^2$ are trivial upper bounds, so the cycle is extremal when
$d=\FL{n/2}$.

The cycle $C_n$ illustrates a broader phenomenon: $\sigma_0(G)=d$ and
$\sigma_1(G)=nd^2$ for any $n$-vertex vertex-transitive graph $G$ with diameter
$d$, attaining the upper bounds.  (A graph $G$ is \emph{vertex-transitive} if
for all $u,v\in V(G)$ there is an automorphism of $G$ mapping $u$ to $v$;
deleting $k$ disjoint edges from $K_{2k}$ yields an example with diameter $2$.)
Many vertex-transitive graphs are known.  The cartesian product of any two
vertex-transitive graphs is vertex-transitive, and its diameter is the sum of
the diameters of the factors.

To generalize $C_n$, define a \emph{$\rho$-stratified graph} to be a graph $F$
having distinguished vertices $x$ and $y$ such that for every $v\in V(F)$, the
distances in $F$ to $x$ and $y$ sum to $\rho$.  Define $Q_{F,k}$ from $2k$
cyclically arranged copies of $F$ with distinguished vertices $x_i$ and $y_i$
in the $i$th copy by merging the vertices $y_i$ and $x_{i+1}$ for all $i$,
where indices are taken modulo $2k$.  The distance from a vertex in $Q_{F,k}$
to the corresponding vertex in the opposite copy of $F$ is $\rho k$, and no
vertex is farther.  Thus although $Q_{F,k}$ is not vertex-transitive (unless
$F$ is a path, which makes $Q_{F,k}$ a cycle), every vertex in $Q_{F,k}$ has
eccentricity $\rho k$.  For the diameter $d$ and number of
vertices $n$ of $Q_{F,k}$, we have $d/n=\rho/(2|V(F)|-2)$.

To achieve a particular ratio $p/(2q)$ for $d/n$, where $q\ge p\ge2$, let $F$
consist of a a path $P$ with $p+1$ vertices plus $q-p$ extra vertices that
are each adjacent to two vertices of $P$ separated by distance $2$.  Thus
$\rho=p$ and $\C{V(F)}=q+1$.  As $p$ increases from $2$ to $q$, these ratios
rise from $1/q$ to $1/2$.  As $q$ increases from $p$, the ratios fall from
$1/2$ toward $0$.

In every such construction, every vertex has the same eccentricity, achieving
the upper bounds for $\sigma_0$ and $\sigma_1$ in Observation~\ref{diambd} and
exceeding the bounds for trees.
\end{Example}

There are no vertex-transitive graphs with $d>n/2$.  The proof also prevents
reaching the upper bound in Observation~\ref{diambd}, since it forbids all
vertices having eccentricity $d$.

\begin{Lemma}\label{diamn2}
Let $G$ be an $n$-vertex connected graph with diameter $d$.
If $d>n/2$, then $\sigma_0(G)<d$ and $\sigma_1(G)<nd^2$.
\end{Lemma}
\begin{proof}
If $\sigma_0=d$ or $\sigma_1=nd^2$, then every vertex must have eccentricity
$d$.  If $G$ has a cut-vertex $v$, then $v$ has smaller eccentricity than a
vertex at maximum distance from $v$, a contradiction.  Hence $G$ must be
$2$-connected.  Now Menger's Theorem implies that any two vertices are
connected by two internally-disjoint paths and hence lie on a cycle.  The cycle
has length at most $n$.  Hence any two vertices are connected by a path of
length at most $n/2$, so $d\le n/2$.
\end{proof}

\begin{Example}\label{treebad}
Although we cannot have eccentricity $d$ at all vertices when $d>n/2$, we can
nevertheless show that trees are not optimal when $d$ is a bit larger than
$n/2$.  For simplicity, we confine our attention to $\sigma_0$.  A similar
argument yields an analogous result for $\sigma_1$.

For $k\le n/2$, let $H_{n,k}$ be the $n$-vertex graph obtained from the cycle
$C_{2k}$ by growing a path from one vertex through $n-2k$ new vertices.  Since
$\diam(H_{n,k})=n-k$, we will compare $\sigma_0(H_{n,k})$ with $\sigma_0(T)$
for $T\in \bT^{n,n-k}$.  We will consider $n\sigma_0$, the sum of the
eccentricities.  From Lemma~\ref{treemax}, $n\sigma_0(T)=nd-\FL{d^2/4}$,
where $d=n-k$.

Let $r=n-2k$, so $d=k+r$.  We consider only $r\le n/3$ and thus $k\ge n/3$ in
order to have a consistent formula for $\sigma_0(H_{n,k})$.  Eccentricities
of vertices along the added path decrease from $k+r$ at the leaf to $k$ at the
endpoint $x$ on the cycle.  The eccentricity of the vertex $y$ opposite $x$ on
the cycle is $k+r$, and the value of the eccentricity decreases moving away
from $y$ in both directions until eccentricity $k$ is reached.  The condition
$k\ge r$ guarantees that this happens at or before $x$ on the cycle.  Each
vertex contributes at least $k$,
and we have $n\sigma_0(H_{n,k})=nk+3\CH{r+1}2-r$.

To compare with the corresponding tree, we write the parameters in terms of
$n$ and $r$.  Since $k=(n-r)/2$, we have $d=(n+r)/2$.  Ignoring the floor
function, $n\sigma_0(T)=\FR{n^2}2+\FR{nr}2-\FR{(n+r)^2}{16}$, and
$n\sigma_0(H_{n,k})=\FR{n^2}2-\FR{nr}2+\FR{3r^2+r}2$.  We thus seek
$r$ between $0$ and $n/3$ such that
$$
\FR{3r^2}2+\FR r2-\FR{nr}2 > \FR{nr}2-\FR{(n+r)^2}{16}.
$$
When we parametrize the problem by setting $r=\alpha n$, the inequality
simplifies to
$$
25\alpha^2-14\alpha+1>\FR{-8\alpha}n.
$$
It suffices to choose $\alpha$ to make the left side positive, and for large
$n$ we cannot do much better.  Solving $25\alpha^2-14\alpha+1>0$, we find that
$H_{n,k}$ has larger average eccentricity than trees in $\bT^{n,n-k}$ when
$0\le\alpha<(7-2\sqrt 6)/25\approx .084$, for sufficiently large $n$.
That is, since the diameter is $(n+r)/2$, when $.5n\le d<.542n$ the non-tree
construction has larger average eccentricity.
\end{Example}

When $d=n-1$, the path $P_n$ is the only example with diameter $d$ and hence is
optimal.  Thus there is some threshold for $d$ in terms of $n$ so that the
trees in $\bT^{n,d}$ are optimal.

\begin{Problem}
Find the least value $d_n$ such that when $d>d_n$, the $n$-vertex graphs with
diameter $d$ having the largest values of $\sigma_0$ and/or $\sigma_1$
are trees.
\end{Problem}

Next we consider upper bounds on $\sigma_2$ for $n$-vertex connected graphs with
diameter $d$.  We restrict to $d\ge2$, since when $d=1$ the only instance is
$K_n$.  Since each edge contributes and we are not stratifying by the number of
edges, we want to generate many edges joining vertices of eccentricity $d$
while maintaining diameter $d$.  Hence we obtain upper bounds by bounding the
number of edges in an $n$-vertex graph with diameter $d$.  This problem was
solved by Ore~\cite{O} in 1968 by determining the structure of
diameter-critical graphs.  An easy direct proof was published by Qiao and
Zhan~\cite{QZ}.

\begin{Lemma}[{\rm \cite{O,QZ}}]\label{ore}
For $d\ge2$, the maximum number of edges in a simple $n$-vertex graph $G$ with
diameter $d$ is $d+(n-d-1)(n-d+4)/2$.  Equality holds if and only if $G$ is
formed from a path $P$ of length $d$ by adding $n-d-1$ vertices that form a
clique and are each adjacent to the first three or the last three among some
three or four consecutive vertices on $P$.
\end{Lemma}

It can be convenient to rewrite the bound $d+(n-d-1)(n-d+4)/2$ as
$\CH{n-d}2+(2n-d-2)$ or $\CH{n-d-1}2+(3n-2d-3)$.
Lemma~\ref{ore} and Observation~\ref{diambd} together give an immediate upper
bound on $\sigma_2(G)$ in this class, and it is asymptotically optimal.

\begin{Example}\label{Bnd}
For diameter $d$ with $d\ge2$, start with a path having vertices $\VEC v0d$
in order.  Construct $B_{n,d}$ by adding $n-d-1$ pairwise-adjacent vertices and
making them all also adjacent to $v_0$ and $v_1$.  Construct $B'_{n,d}$ from
the path by adding $n-d-1$ pairwise-adjacent vertices and making them all also
adjacent to $v_0$, $v_1$, and $v_2$.  Note that $B_{n,d}$ can also be obtained
by growing a path from one vertex of a complete graph; from this viewpoint it
was called a \emph{kite} in~\cite{DMCC}.  We use the notation $B_{n,d}$ to
emphasize the diameter and the relationship to \emph{brooms}, which are trees
obtained by adding leaves at one end of a path.

Both $B_{n,d}$ and $B'_{n,d}$ have diameter $d$.  In $B'_{n,d}$ there are
$n-d-1$ more edges, but its added vertices have eccentricity $d-1$ instead of
$d$.  With $g_2(d)$ defined as before Lemma~\ref{treemax},
\begin{align*}
\sigma_2(B_{n,d})&= g_2(d)+(n-d-1)d(d-1)+\CH{n-d}2d^2\qquad\mbox{for~$d\ge2$},\\
\sigma_2(B'_{n,d})&=
g_2(d)+(n-d-1)d(d-1)+\CH{n-d}2(d-1)^2+(n\!-\!d\!-\!1)(d\!-\!1)(d\!-\!2)
\qquad\mbox{for~$d\ge4$}.
\end{align*}
Since eccentricity increases when the middle of a diametric path is passed,
for $d\le3$ the formula for $\sigma_2(B'_{n,d})$ must be adjusted.
For $d=3$ the last term is replaced by $(n-d-1)(d-1)^2$, and for $d=2$ it
is replaced by $(n-d-1)(d-1)d$.  Those adjustments are not important, because
$\sigma_2(B_{n,3})>\sigma_2(B'_{n,3})$ for all $n$, and the extremal graph when
$d=2$ is clearly the graph obtained by deleting a perfect matching from $K_n$
(when $n$ is odd, one vertex loses two edges).

For $4\le d\le n-2$, we compute
\begin{align*}
\sigma_2(B_{n,d})-\sigma_2(B'_{n,d})
&= \FR{n-d-1}2[(n-d)(2d-1)-2(d-1)(d-2)]\\
&= \FR{n-d-1}2[(n-d)+(2d-2)(n-2d+2)].
\end{align*}
When $d\le(n+2)/2$, the second factor is positive and $\sigma_2(B_{n,d})$ is
larger, but when $d\ge(n+3)/2$ it is negative and $\sigma_2(B'_{n,d})$ is
larger.  This is natural, since the advantage of higher eccentricity for the
added vertices diminishes as $d$ grows and there are fewer of them.
\end{Example}

We combine Example~\ref{Bnd} with Lemma~\ref{ore} to obtain the asymptotics
of the solution.

\begin{Theorem}
For any $n$-vertex graph connected $G$ with diameter $d$,
$$
\sigma_2(G)\le \CH{n-d}2d^2+2(n-d-1)d^2+d^3.
$$
This upper bound is asymptotically sharp when $d$ is bounded by any constant
fraction of $n$, as shown by $B_{n,d}$ and $B'_{n,d}$ in Example~\ref{Bnd}.
\end{Theorem}
\begin{proof}
We rewrite the bound in Lemma~\ref{ore} by extracting the quantity $\CH{n-d}2$,
then we multiply by $d^2$ to apply Observation~\ref{diambd}.  This establishes
the upper bound.

The formula featuring $\CH{n-d}2$ facilitates comparison with Example~\ref{Bnd}.
Recall that $g_2(d)=(7/12)d^3+O(d)$.  In comparing the leading terms of the
upper bound here with the values for the constructions in Example~\ref{Bnd},
the upper bound is larger by about $(n-d-1)d^2+(5/12)d^3$.  Since the leading
behavior is quartic in $d$ when $d$ is a constant fraction of $n$, the
constructions in Example~\ref{Bnd} are asymptotically optimal in that range.
\end{proof}

\begin{Corollary}\label{s2alln}
Over $n$-vertex connected graphs, the maximum of $\sigma_2$ is $n^4/32+O(n^3)$.
\end{Corollary}
\begin{proof}
Using the expression in Lemma~\ref{ore}, the upper bound on $\sigma_2$ for
$n$-vertex graphs with diameter $d$ is $d^3+d^2(n-d-1)(n-d+4)/2$.  Choosing
$d=\CL{n/2}+2$ to maximize the bound yields $\sigma_2\le n^4/32+O(n^3)$, and
with that value of $d$ this is achieved by $B_{n,d}$ or $B'_{n,d}$.
\end{proof}

Our goal in the remainder of this section is to show that the maximum of
$\sigma_2$ over $n$-vertex graphs with diameter $d$ is always achieved by
$B_{n,d}$ or $B'_{n,d}$ when $d\ge3$.  We begin with a family of constructions
that move from $B_{n,d}$ to $B'_{n,d}$ by adding one edge with each step.

\begin{Example}  \label{BT}
For $0\le t\le n-d-1$, let $B_{n,d,t}$ be the graph obtained from $B_{n,d}$ by
making $t$ of the $n-d-1$ vertices in the complete subgraph outside
$\{\VEC v0d\}$ adjacent to $v_2$.  Note that $B_{n,d,0}$ is $B_{n,d}$ and
$B_{n,d,n-d-1}$ is $B'_{n,d}$.  We move from $B_{n,d,t}$ to $B_{n,d,t+1}$ by
adding one edge $xv_2$, which contributes $(d-1)(d-2)$ to the sum but reduces
the eccentricity of $x$ from $d$ to $d-1$.  This costs $d$ (or $d-1$) from the
contribution for each edge from $x$ to a neighbor with eccentricity $d$ (or
$d-1$, respectively).  In $B_{n,d,t}$, vertex $x$ has $t+1$ neighbors with
eccentricity $d-1$ and $n-d-t-1$ neighbors with eccentricity $d$.  Hence
\begin{align*}
\sigma(B_{n,d,t+1})-\sigma(B_{n,d,t})
&=(d-1)(d-2)-(t+1)(d-1)-(n-d-1-t)d\\
&= 2d^2-(n+3)d+t+3.
\end{align*}
For fixed $n$ and $d$, let $f(t)$ denote this difference.
If $d\le(n+2)/2$, then $f(t)<0$ for $t\le n-d-2$;
if $d\ge(n+3)/2$, then $f(t)>0$ for $t\ge0$.
Thus $\sigma_2$ changes monotonically as $B_{n,d,t}$ runs from $B_{n,d}$ to
$B'_{n,d}$, with $\sigma_2(B_{n,d})>\sigma_2(B'_{n,d})$ when $d\le(n+2)/2$ and
$\sigma_2(B_{n,d})<\sigma_2(B'_{n,d})$ when $d\ge(n+3)/2$, as computed
explicitly in Example~\ref{Bnd}.

The monotonicity property implies that always
$\sigma_2(B_{n,d,t})\le \max\{\sigma_2(B_{n,d}),\sigma_2(B'_{n,d})\}$.
\end{Example}

In order to show that $\sigma_2(G)$ for an $n$-vertex graph $G$ with
diameter $d$ is bounded by $\sigma_2(B_{n,d,t})$ for some $t$ and hence
by $\sigma_2(B_{n,d})$ or $\sigma_2(B'_{n,d})$, we will need several lemmas.
The first is essentially Lemma 4 in \cite{HHMRD}.  We include a proof for
completeness.

\begin{Lemma}\label{fact1}
Let $G$ be an $n$-vertex graph with diameter $d$, where $n>d\ge3$, and
let $P$ be a diametric path in $G$.  Let $u$ be a vertex of $P$ such that
$\eps(u)$ exceeds the maximum distance $\ell$ from $u$ to an endpoint of $P$,
and let $v$ be a vertex at distance $\eps(u)$ from $u$.  Let $P'$ be a path of
length $\eps(u)$ from $v$ to $u$, with vertices $\VEC w1{\eps(u)+1}$ in order
such that $v=w_1$ and $u=w_{\eps(u)+1}$.  Under these conditions,

(a) $\VEC w1{\eps(u)-\ell}$ do not belong to $P$.

(b) $\VEC w1{\eps(u)-\ell-1}$ have no neighbor on $P$, and the only neighbor
$w_{\eps(u)-\ell}$ can have on $P$ is one endpoint of $P$ most distant from $u$.
\end{Lemma}
\begin{proof}
If some $w_i$ with $i\le \eps(u)-\ell$ lies on $P$, then $G$ has a path from
$u$ to $v$ through $w_i$ with length at most $\ell+\eps(u)-\ell-1$,
contradicting $d_G(u,v)=\eps(u)$.  If $i<\eps(u)-\ell$ and $w_i$ has a neighbor
on $P$, then the same computation applies.  When $i=\eps(u)-\ell$, to avoid
this argument the neighbor of $w_i$ on $P$ must be at distance $\ell$ from $u$,
which makes it an endpoint of $P$.  Since $d\ge3$, when both endpoints of $P$
have distance $\ell$ from $u$ it is not possible for $v$ to be adjacent to both
of them.
\end{proof}

Let $P$ be an arbitrary diametric path in the graph $G$.
For $w\notin V(P)$, let $d_1(w)$ be the number of neighbors of $w$ on $P$,
and let $d_2(w)$ be the number of neighbors of $w$ outside $P$.

\begin{Theorem}\label{main}
If $n>d\ge3$, then $\sigma_2$ is maximized over $n$-vertex graphs with diameter
$d$ by $B_{n,d}$ (when $d\le(n+2)/2$) or $B'_{n,d}$ (when $d\ge(n+3)/2$).
\end{Theorem}
\begin{proof}
Let $P$ with vertices $\VEC u0d$ be a diametric path in an $n$-vertex graph $G$
with diameter $d$.  For $0\le i\le d$, let $\delta_i=\max\{i,d-i\}$.  The
vertex $u_i$ is at distance $\delta_i$ from the farthest end of $P$.  If
$\eps_G(u_i)>\delta_i$, then we say that $u_i$ has {\it excess eccentricity}.
We consider two cases to prove
$\sigma_2(G)\le \max\{\sigma_2(B_{n,d}),\sigma_2(B'_{n,d})\}$.

\medskip
{\bf Case 1:} {\it $P$ has no vertex with excess eccentricity.}
The contribution of $P$ to $\sigma_2(G)$ is $g_2(d)$.
Let $m$ be the number of edges of $G$ not on $P$.

Vertices outside $P$ having eccentricity $d$ in $G$ can only be adjacent to
each other and to the last two vertices at one end of $P$ or the other
(not both).  Vertices with eccentricity at most $d-1$ have at most three
neighbors on $P$, and they may also be adjacent to each other or to the
vertices with eccentricity $d$.

Let $t$ be the number of vertices outside $P$ having eccentricity at most $d-1$
in $G$.  The contribution to $\sigma_2(G)$ from edges with both endpoints
outside $P$ is at most what it is in $B_{n,d,t}$, since those vertices
form a clique in $B_{n,d,t}$, and corresponding vertices have eccentricities at
least as large in $B_{n,d,t}$.

Similarly, for edges with one endpoint on $P$, the contribution for those
having a specified outside endpoint with eccentricity $d$ is at most
$d(d+d-1)$, which is what it equals in $B_{n,d,t}$.  For those having one
of the $t$ outside vertices having eccentricity at most $d-1$ it is at most
$(d-1)(d+d-1+d-2)$, which is what it equals in $B_{n,d,t}$.

Since each group of edges contributes at least as much to $\sigma_2(B_{n,d,t})$
as to $\sigma_2(G)$, we conclude $\sigma_2(G)\le \sigma_2(B_{n,d,t})$.
Thus the claim follows from the monontonicity property in Example~\ref{BT}.

\medskip
{\bf Case 2:} {\it $P$ has a vertex with excess eccentricity.}
Our goal in this case is to modify $G$ into a new graph $G'$ such that
$G'$ also has diameter $d$ and $n$ vertices and satisfies
$\sigma_2(G')\ge\sigma_2(G)$ and $G'\esub B_{n,d,t}$ for some $t$.
By Example~\ref{BT}, this will complete the proof.  In some subcases we will
further need to change $G'$ into $G''$ by adding some edges to achieve
$\sigma_2(G'')\ge\sigma_2(G)$ and $G''\esub B_{n,d,t}$.

Let $V_0$ be the set of vertices in $P$ with excess eccentricity.
Let $r_i=\eps_G(u_i)-\delta_i$ and Let $r^*=\max_i\{r_i\}$, so $r^*$ is the
maximum excess eccentricity.  Let $u_k$ be some vertex in $V_0$
among those with excess eccentricity $r^*$.  Diameter $d$ forbids the endpoints
of $P$ from $V_0$, so by symmetry we may assume that $P$ is indexed with
$k\in\{1,\ldots,d/2\}$.

Let $v$ be a vertex with distance $\eps_G(u_k)$ from $u_k$; note that
$v\notin V(P)$.  Let $\la\VEC w1{\eps_G(u_k)+1}\ra$ be a shortest path from $v$
to $u_k$.  By Lemma~\ref{fact1}, $d_1(w_j)=0$ for $j<r^*$ and
$d_1(w_{r^*})\le1$.  Let $S=V(G)-V(P)$, and let $S_0=\{\VEC w1{r^*}\}\esub S$.

We form $G'$ from $G$ by leaving $P$ and the edges induced by $S$ unchanged,
but deleting all edges from $S$ to $V(P)$ and replacing them as follows.
For $w\in S$, the neighbors of $w$ in $G'$ are $\{u_0,u_1\}$, except that when
$d_1(w)=3$ and $\eps_G(w)<d$ we also make $w$ adjacent to $u_2$.
$$
N_{G'}(w)\cap V(P)=
\begin{cases} \{u_0,u_1,u_2\}& \mbox{if $d_1(w)=3$ and $\eps_G(w)<d$},\\
              \{u_0,u_1\}& \mbox{otherwise}.  \end{cases}
$$

Note that $\eps_{G'}(u_i)=\delta_i$ for all $i$.  For $w\in S$, we have
$\eps_{G'}(w)=d$ unless $d_1(w)=3$ and $\eps_G(w)<d$, in which case
$\eps_{G'}(w)=d-1$.  Thus $G'$ is a subgraph of $B_{n,d,t}$ with all
vertices having the same eccentricity as in $B_{n,d,t}$, where $t$ is
the number of vertices of $S$ satisfying $d_1(w)=3$ and $\eps_G(w)<d$.

Furthermore, $\eps_{G'}(w)\ge \eps_G(w)$ for $w\in S$, so the edges induced
by $S$ contribute at least as much to $\sigma_2$ in $G'$ as in $G$.
However, there are two ways $\sigma_2$ may decrease.
Due to the excess eccentricity in $G$ of vertices along $P$, the contribution
of $E(P)$ to $\sigma_2$ is smaller in $G'$ than in $G$.  In Step 1 we will
use edges joining $S_0$ to $V(P)$ to overcome this loss.  Meanwhile,
the edges joining a vertex $v$ of $S-S_0$ to $V(P)$ may make smaller
contributions to $\sigma_2$ in $G'$ than in $G$.  To overcome this loss,
in Step 2 we will add edges joining such vertices $v$ to other vertices of
$S$, forming $G''$.

\medskip
{\bf Step 1:} {\it The loss $D$ in the contribution along $P$ is at most
the gain from edges joining $S_0$ to $V(P)$, with possible help needed from one
more vertex and an exception when d=4.}
Using
$r_0,r_d=0$ and $r_1,r_{d-1}\le 1$ and $\delta_1=\delta_{d-1}$, we compute
\begin{align*}
D&=\SE i0{d-1} [\eps_G(u_i)\eps_G(u_{i+1})-\eps_{G'}(u_i)\eps_{G'}(u_{i+1})]
~=~\SE i0{d-1} [(\delta_i+r_i)(\delta_{i+1}+r_{i+1})-\delta_i\delta_{i+1}]\\
&=dr_1+dr_{d-1}+\SE i1{d-2}(\delta_ir_{i+1}+r_i\delta_{i+1}+r_ir_{i+1})
~\le~ 2d+\SE i1{d-2}(\delta_i+\delta_{i+1}+r^*)r^*\\
&=2d+(d-2)(r^*)^2+2r^*\SE i1{d-2}\delta_i
~=~2d+(d-2)(r^*)^2+2r^*(g_0(d)-(3d-1))\\
&=2d+(d-2)(r^*)^2+2r^*\CL{\FR{3d^2-8d+4}4}.
\end{align*}
In the special case where $r_i\le 1$ for
$i\le\FL{d/2}$ and $r_i=0$ for $i\ge\FL{d/2}+1$,
\begin{align*}
D&\le\SE i0{\FL{d/2}}(\delta_ir_{i+1}+r_i\delta_{i+1}+r_ir_{i+1})
=d+(\FL{d/2}+1)+\SE i1{\FL{d/2}-1}(\delta_i+\delta_{i+1}+1)
\\
&=\begin{cases} \FR{3d^2+4}4 & \mbox{if $d$ is even}\\
                \FR{3d^2-1-2d}4 &\mbox{if $d$ is odd.}  \end{cases}
\end{align*}

The contribution to $\sigma_2(G)$ for edges from $S_0$ to $V(P)$ is $0$
if $d_1(w_{r^*})=0$ and is at most $d^2$ if $d_1(w_{r^*})=1$, by
Lemma~\ref{fact1}.  Let $W$ be the contribution to $\sigma_2(G')$ for edges
from $S_0$ to $V(P)$.  These $r^*$ vertices each contribute $d(2d-1)$ for the
edges to $u_0$ and $u_1$.  We compute
\begin{align*}
W-D&\ge d(2d-1)r^* - \FR{3d^2-8d+5}2 r^*-(d-2)(r^*)^2-2d\\
&= \FR{d^2}2r^*-d(r^*)^2+(3d-5/2)r^*+2(r^*)^2-2d\\
&=dr^*(d/2-r^*)+(3d-5/2)r^*+2(r^*)^2-2d.
\end{align*}
Since $1\le r^*\le d/2$ and $d\ge3$, the last expression $X$ above is positive.
This suffices when $d_1(w_{r^*})=0$.  When $d_1(w_{r^*})=1$, we need
$W-D\ge d^2$ to overcome the loss along $P$ using only the edges from $S_0$,
since the edge from $w_{r^*}$ to $u_d$ may already contribute as much as
$d^2$ to $\sigma_2(G)$.  Since $X$ is quadratic in $r^*$ with leading
coefficient $2-d$, we need only check the extreme values of $r^*$.

When $r^*=d/2$, we have $X=2d^2-13d/4$, which is at least $d^2$ when $d\ge4$.
When $r^*=(d-1)/2$, we have $X=(9/4)d^2-6d+7/4$; this value exceeds $d^2$ when
$d\ge5$.  When $d=3$ we have $(d-1)/2=1$ and apply the special argument
for $r^*=1$ below.  When $r^*=2$, we have $X=d^2+3$, which is big enough.

When $r^*=1$, we only get $X=(d^2-1)/2$, so for the case $r^*=1$ and
$d_1(w_{r^*})=1$ we need more help.  If there are at least two vertices that
are most distant from vertices on $P$ with excess eccentricity $1$, then choose
two of these to form a set $S'_0$.  We treat these like $S_0$, considering
the edges from both to $V(P)$.  Each member of $S'_0$ contributes $d(2d-1)$ in
$G'$, for a total contribution $W'$ equal to $2d(2d-1)$.  Since each member of
$S'_0$ may be adjacent in $G$ to one endpoint of $P$ via an edge contributing
up to $d^2$ to $\sigma_2(G)$, we need to gain at least $2d^2$.
Setting $r^*=1$ in our bound on $D$, we compute
$$
W'-D\ge 2d(2d-1)-(3d^2-2d+1)/2 = 5d^2/2-d-1/2>2d^2.
$$

Now suppose that $w$ is the only vertex that can serve as $w_1$ for
vertices on $P$, and $d_1(w)=1$.  We may assume that $w$ is adjacent to
$u_d$, and hence it is within distance $\delta_i$ of $u_i$ for each $i$ with
$i\ge\FL{d/2}+1$.  Thus for $i\ge\FL{d/2+1}$ there is no vertex with distance
more than $\delta_i$ from $u_i$, so $r_i=0$ for $i\ge\FL{d/2}+1$.  Using the
improved bound on $D$ for this case, we compute
$$
W-D\ge d(2d-1)-\FR{3d^2}4+
\begin{cases} -1& \mbox{if $d$ is even}\\
              (2d+1)/4 &\mbox{if $d$ is odd}  \end{cases} .
$$
The lower bound is at least $d^2$ for $d\ge3$
as desired, except that when $d=4$ the inequality fails by $1$,
with the lower bound being only $15$.  We postpone fixing this $1$.

\medskip
{\bf Step 2:} {\it The losses to $\sigma_2$ for edges from $S$ to $V(P)$
at vertices $w$ in $S$ such that $d_1(w)\ge2$ can be overcome by adding edges
within $S$.}
For $w\in S$, let $\Gamma(w)$ denote the set of edges joining $w$ to $V(P)$,
in $G$ or $G'$ as appropriate.
Thus far we have considered $\Gamma(w)$ only for $w\in S_0$, or possibly at two
vertices in $S'_0$, all of which satisfy $d_1(w)\le1$.  Any other $w\in S$
with $d_1(w)\le 1$ also causes no difficulty, since $\Gamma(w)$ for such $w$
contributes at most $d^2$ to $\sigma_2(G)$ and exactly $d(2d-1)$ to
$\sigma_2(G')$.

When considering $\Gamma(w)$ for $w\in S$ with $d_1(w)\ge2$, moving from $G$ to
$G'$ can produce losses to $\sigma_2$.  We will overcome the loss by giving $w$
one or two new neighbors in $S$.  The result of doing this for all vertices
with loss is $G''$.  Each edge of the form $wz$ that we add to cover a loss at
$w$ has the property that $d_1(z)<d_1(w)$, so we never introduce an edge twice.
The result will be $\sigma_2(G'')> \sigma_2(G)$, and still $G''\esub B_{n,d,t}$.

Since $P$ is a diametric path, always $d_1(w)\le3$.  Consider $w\in S$ with
$d_1(w)\ge2$.  Suppose that some neighbor $u$ of $w$ on $P$ has excess
eccentricity in $G$.  By Lemma~\ref{fact1}, there is a vertex $w'\in S$ at
distance $\eps_G(u)$ from $u$, and $w'$ has at most one neighbor on $P$.  Since
$d\ge3$ and $\eps_G(u)\ge3$ (due to $d\ge3$ and having excess eccentricity), we
have $ww'\notin E(G)$.  Also $\eps_{G'}(w)\ge d-1$ and $\eps_{G'}(u)=d$.  Hence
adding $ww'$ to $G'$ adds at least $d^2-d$ to $\sigma_2$.

If $d_1(w)=2$, then $\Gamma(w)$ contributes $d(2d-1)$ in $G'$, which is a loss
if and only if $\Gamma(w)$ contributes $2d^2$ in $G$, which requires $w$ to
have a neighbor on $P$ with excess eccentricity.  Now the edge $ww'$
contributes enough to overcome the loss.

Hence we may assume $d_1(w)=3$.  First suppose $\eps_G(w)\le d-1$.
Here $\Gamma(w)$ contributes $(d-1)(3d-3)$ to $\sigma_2(G')$.
If $w$ has no neighbor on $P$ with excess eccentricity, then $\Gamma(w)$
also contributes at most $(d-1)(3d-3)$ to $\sigma_2(G)$ (if $d\ge4$),
and there is no loss.  If $w$ does have a neighbor with excess eccentricity,
then $\Gamma(w)$ may contribute up to $(d-1)(3d)$ to $\sigma_2(G)$, yielding
a loss of at most $3(d-1)$.  However, in this case the edge $ww'$ discussed
above is available, contributing at least $d^2-d$ to overcome the loss,
since $d^2-d\ge 3(d-1)$ when $d\ge3$.  (In the special case where $d=3$ and
$w$ has no neighbor on $P$ with excess eccentricity, the contribution by
$\Gamma(w)$ is $(d-1)(3d-2)$ to both $\sigma_2(G)$ and $\sigma_2(G')$, with
no loss.)

The remaining case is $d_1(w)=3$ and $\eps_G(w)=d$.  In this case all vertices
on $P$ are within distance $d-1$ of $w$.  Hence a vertex $v$ at distance $d$
from $w$ must lie in $S$, and avoiding a path from $w$ to $v$ with length less
than $d$ requires $d_1(v)\le 1<d_1(w)$.  Thus $\eps_{G'}(v)=d$, and adding the
edge $wv$ to $G'$ will add $d^2$ to $\sigma_2$.  Since $\Gamma(w)$ contributes
$d(2d-1)$ to $\sigma_2(G')$ and at most $d(3d)$ to $\sigma_2(G)$, the addition
of $d^2$ for the edge $wv$ overcomes the loss unless all three neighbors of
$w$ have eccentricity $d$ in $G$.

Let the neighbors of $w$ on $P$ be $\{u_{j-1},u_j,u_{j+1}\}$.  Since $d\ge3$,
one of $\{u_{j-1},u_{j+1}\}$ is not an endpoint of $P$; by symmetry, we may let
this be $u_{j+1}$.  If $v$ has a neighbor on $P$, then $d_G(u_{j+1},v)\le d-1$.
Now a vertex $y$ at distance $d$ from $u_{j+1}$ in $G$ is both in $S$ and
different from $v$.  Since $d\ge3$, we have $wy\notin E(G)$.  Again $y$ has at
most one neighbor on $P$, by Lemma~\ref{fact1}, and we gain $2d^2$ by adding
both $wv$ and $wy$ to $G''$.

On the other hand, if $v$ has no neighbor on $P$, then the vertex $x$ reached
before $v$ on a path of length $d$ from $w$ is not on $P$.  Since $d\ge3$, we
have $wx\notin E(G)$.  Since $\eps_{G'}(x)\ge d-1$, adding $wv$ and $wx$ is
enough to gain $2d^2-d$, which is enough to overcome the loss of
$d^2+d$ due to $\Gamma(w)$.  Finally, to avoid having a path of length at most
$d-2$ connecting $w$ and $x$ using edges of $P$, it cannot happen that $w$ and
$x$ both have three neighbors on $P$.  Thus $d_1(x)<d_1(w)$, so this added edge
has not been used also to cover a possible loss at $x$.

\medskip
{\bf Step 3:} {\it The leftover case $d=4$ from Step 1.}
At the end of Step 1 there was one case where we failed by $1$ to overcome
the loss $D$, with $W-D=-1$.  In this case $d=4$ and $r^*=1$, the vertices
$\VEC u04$ along $P$ have eccentricities $4,4,3,3,4$ (respectively), there
is only one vertex $w$ that is at maximum distance from the vertices $u_1$
and $u_2$ that have excess eccentricity, and $d_1(w)=1$, with $w$ adjacent to
$u_4$.

If forming $G'$ fails to gain $1$ to overcome the loss, then in addition
we have $d_1(v)\ge2$ for all $v\in S$ other than $w$.  Because $d_G(w,u_1)=4$,
the path of length at most $4$ from $w$ to $u_0$ must move from $w$ to
$x_1$ to $x_2$ with $x_1,x_2\in S$.  Because this is a shortest path,
$wx_2\notin E(G)$.  We claim that adding the edge $wx_2$ to $G'$ will overcome
both the loss of $1$ in $W-D$ and any loss there may be from $\Gamma(x_2)$.

If $d_1(x_2)=2$, then $\eps_{G'}(x_2)=d$, and $wx_2$ contributes $d^2$.  This
suffices, since the maximum loss for $\Gamma(x_2)$ is $d$ when $d_1(x_2)=2$.
Hence we may assume $d_1(x_2)=3$.  Since $d_G(w,u_1)=4$, the three consecutive
neighbors of $x_2$ on $P$ must be $\{u_2,u_3,u_4\}$.  If $\eps_G(x_2)\le d-1$,
then the loss for $\Gamma(x_2)$ is at most $d-1$, overcome by $wx_2$.
If $\eps_G(x_2)=d$, then $\Gamma(x_2)$ contributes $40$ (that is, $d(3d-2)$)
to $\sigma_2(G)$ and $28$ (that is, $d(2d-1)$) to $\sigma_2(G')$.  Now
the contribution of $16$ (that is, $d^2$) from the edge $wx_2$ is enough
to overcome both losses.
\end{proof}

In fact, the extremal graph is unique, either $B_{n,d}$ or $B'_{n,d}$,
depending on the relationship between $d$ and $n$, as discussed in
Example~\ref{Bnd}.  This follows from the process in Step 2 of
Theorem~\ref{main} if the graph we produce exhibits a strict gain, is a
subgraph of $B_{n,d,t}$ with $1\le t\le n-d-2$, or is a proper subgraph of
$B_{n,d}$ or $B'_{n,d}$.  We omit the detailed verification that one of
these possibilities always occurs.

\section{Lower bounds in terms of diameter}
We now consider lower bounds in terms of diameter.  For $\sigma_0$ and
$\sigma_1$ these are easy, because extra edges impose no cost.  The final
formula in the result for $\sigma_1$ is a minor rephrasing of that obtained by
Das et al.~\cite{DLG} by the same argument.  We give a short proof.

\begin{Proposition}[{\rm \cite{DLG} for $\sigma_1$}]\label{d01low}
If $G$ is an $n$-vertex graph with diameter $d$, then
$$
\sigma_0(G)\ge \FR1n\left[g_0(d)+(n-d-1)\CL{\FR d2}\right]
=\CL{\FR d2}+\FR1n\FL{\FR d2}\CL{\FR{d+1}2}
$$
and
$$
\sigma_1(G)\ge g_1(d)+(n-d-1)\CL{\FR d2}^2
=n\CL{\FR d2}^2+\FR{d^3}3+\FR1{12}
  \begin{cases} 9d^2+2d&   \mbox{if $d$ is even}\\
                3d^2-4d-3& \mbox{if $d$ is odd.} \end{cases}
$$
Equality is achieved by a graph consisting of a path $P$ of length $d$ plus
$n-d-1$ vertices with eccentricity $\CL{d/2}$ whose neighborhoods are the
same as a vertex with eccentricity $\CL{d/2}$ on $P$.  When $d$ is odd, one can
instead make some added vertices adjacent to the two central vertices on the
path.  Additional edges joining vertices with eccentricity $\CL{d/2}$ do
not change $\sigma_0$ or $\sigma_1$.
\end{Proposition}
\begin{proof}
Let $P$ be a diametric path with endpoints $x$ and $y$.  For $v\in V(P)$, the
sum of the distances to $x$ and $y$ must be $d$.  Hence the eccentricities of
vertices along $P$, in order, are at least $d,d-1,\ldots,d-1,d$.  In addition,
the eccentricity is at least $d/2$ for any vertex outside $P$.  Summing these
contributions yields the claimed lower bounds.

For sharpness, note that for $d=1$ the only example is $K_n$.  For even $d$,
add $n-d-1$ vertices outside $P$ as copies of the unique vertex $u\in V(P)$ with
eccentricity $d/2$, adjacent to the neighbors of $u$ on $P$.  When $d$ is
odd and $d>1$, one can make the added vertices adjacent to the two central
vertices on $P$ or to both neighbors of one of those vertices, giving each
added vertex eccentricity $\CL{d/2}$.  This gives vertices the eccentricities
described in the lower bound.  Finally, in either case one can add edges
joining any vertex that both have eccentricity $\CL{d/2}$.
\end{proof}

The problem of minimizing $\sigma_2$ is more delicate, because we must now pay
attention to the number of edges joining vertices with small eccentricity.  We
consider first the more refined problem of minimizing $\sigma_2$ over
$n$-vertex graphs when the diameter and number of edges are both fixed.  Later
we use that to minimize over $n$-vertex graphs with diameter $d$.

Using that every vertex in a graph with diameter $d$ has eccentricity at least
$d/2$, Das et al.~\cite{DLG} observed $\sigma_2(G)\ge g_2(d)+(m-d)\CL{d/2}^2$
for every $n$-vertex graph with $m$ edges having diameter $d$.  They asserted
that equality holds using graphs described in Proposition~\ref{d01low}, but in
general this is not correct.  The result was corrected by Hayat~\cite{H} (with
some minor typos we correct).  We state the lower bounds differently from Hayat
in order to clarify the role of various contributions to the bound, and we
present a simpler proof.  We use Hayat's rephrasing of the bound on the number
of edges from Lemma~\ref{ore}.

\begin{Theorem}[{\rm \cite{H}}]\label{dmlow}
Let $G$ be a connected graph with $n$ vertices, $m$ edges, and diameter $d$.
Note that $n-1\le m\le \CH{n-d}2+(2n-d-2)$.

If $d$ is odd and $2n-2-d\le m\le\CH{n-d}2+n-1$, then
$$
\sigma_2(G)\ge \FR{7d^3-d+6}{12}+(m-d)\FR{(d+1)^2}4.
$$
For $m=2n-2-d-k$ or $m=\CH{n-d}2+n-1+k$ with $1\le k\le n-d-1$, the lower bound
increases by $k(d+1)/2$.

If $d$ is even with $d\ge4$ and $\mu_{k-1}<m\le \mu_k$, where
$\mu_j=n-1+\CH j2+2j$, then
$$
\sigma_2(G)\ge \FR{7d^3-4d}{12}+(m-d)\FR{d^2}4+(n-1-d+k)\FR d2.
$$
If $d=2$ and $\nu_k\le m<\nu_{k+1}$, where $\nu_j=\CH j2+j(n-j)$,
then $\sigma_2(G)\ge \CH k2+2k(n-k)+4(m-\nu_k)$.

The lower bounds are sharp in all cases, except that when $d=4$ and
$m=n$ with $n\ge7$, the lower bound must be increased by $3$.
\end{Theorem}
\begin{proof}
The edges along a fixed diametric path $P$ contribute at least $g_2(d)$ to
$\sigma_2(G)$.  Since every vertex has eccentricity at least $\CL{d/2}$, we
obtain $\sigma_2(G)\ge g_2(d)+(m-d)\CL{d/2}^2$.
In each case we describe adjustments to the lower bound and a construction
achieving equality.

{\bf Observation:}
Let $P$ have endpoints $x$ and $y$.  Let $X$ and $Y$, respectively, be the sets
of vertices within distance $\CL{d/2}-1$ of $x$ and $y$.  In order to have
eccentricity $\CL{d/2}$, a vertex not in $V(P)$ must have neighbors in both $X$
and $Y$.  Furthermore, $X$ and $Y$ are disjoint.

\medskip
{\bf Case 1:} \emph{$d$ is odd.}
When $d=1$, the only graph is $K_n$, and the formula yields the correct value,
so we may assume $d\ge3$.
By the observation, having cost $\CL{d/2}^2$ for every edge outside $P$
requires the number of edges outside $P$ to be at least twice the number of
vertices outside $P$, so $m\ge d+2(n-1-d)=2n-d-2$.  In addition, there are
at most $n-d+1$ vertices with eccentricity $\CL{d/2}$ (counting two on $P$), so
achieving this bound requires $m\le d-1+\CH{n-d+1}2=\CH{n-d}2+n-1$.

When $m=\CH{n-d}2+n-1+k$ with $k\ge1$, each additional edge beyond
$\CH{n-d}2+n-1$ costs at least $\CL{d/2}(\CL{d/2}+1)$ and thus incurs a penalty
of at least $\CL{d/2}$ compared to the bound.  Equality can be achieved,
because adding an edge from a vertex outside $P$ to a neighbor on $P$ of a
central vertex costs $\CL{d/2}(\CL{d/2}+1)$ per edge until the construction for
the maximum number of edges in an $n$-vertex graph with diameter $d$ (as
described in Lemma~\ref{ore}) is reached.

When $m=2n-2-d-k$ with $k\ge1$, we do not have enough edges to give two to
each vertex outside $P$.  At least $k$ of these vertices contribute only one
edge as $G$ is grown from $P$, which by the observation gives these vertices
eccentricity at least $\CL{d/2}+1$.  Hence the edges reaching them each
contribute at least $\CL{d/2}(\CL{d/2}+1)$, each incurring a penalty of at
least $\CL{d/2}$ compared to the original lower bound.  Equality is achieved
by the graph consisting of $P$ plus $n-1-d-k$ vertices adjacent to the two
central vertices of $P$ and $k$ vertices adjacent to just one of the two
central vertices of $P$.

\medskip
{\bf Case 2:} \emph{$d$ is even and $d\ge4$, with $m\notin\{n,n-2\}$ when
$d=4$.}
When $d$ is even, every vertex in $X\cup Y$ has eccentricity at least $d/2+1$,
since $d_G(x,y)=d$.  From the observation, we conclude that a vertex with
eccentricity $d/2$ outside $P$ has at least two incident edges that contribute
at least $(d/2)(d/2+1)$ rather than $d^2/4$ to $\sigma_2(G)$.  This incurs a
penalty of $d$ compared to $g_2(d)+(m-d)d^2/4$ for each vertex with
eccentricity $d/2$ outside $P$.  Edges among such vertices can be added with no
additional penalty.

If $m$ is not too big, then one can alternatively add vertices with
eccentricity $d/2+1$ by making them adjacent only to the central vertex of $P$.
The penalty for growing such a vertex $v$ is $d/2$ rather than $d$.  Any
additional edge incident to $v$ will incur additional penalty; if that
additional penalty is only $d/2$, then $v$ is adjacent to another vertex with
eccentricity $d/2$, and it more helpful to make $v$ one of the vertices with
eccentricity $d/2$ as described above.  (The construction does not work when
$d=4$ and $m\in\{n,n+2\}$; see Case 4.)

Hence with $k$ vertices of eccentricity $d/2$ and $n-1-d-k$ vertics of
eccentricity $d/2+1$ outside $P$, the penalty above the initial lower bound
will be $(n-1-d+k)d/2$.  With all vertices outside $P$ adjacent to the central
vertex of $P$, the number of edges that can be accommodated and
avoid larger penalty is $n-1+\CH{k}2+2k$.

Hence the amount to be added to the lower bound $g_2(d)+(m-d)d^2/4$ is
$(n-d-1)d/2$ when $m=n-1$ and rises monotonically from that to $(n-d-1)d$
as $m$ grows to $d+\CH{n-d-1}2+2(n-d-1)$.  It depends on how many vertices
of degree $d/2$ are needed to accommodate the edges.  When $m>\mu_{n-1-d}$,
no $n$-vertex graph with diameter $d$ has $m$ edges; $\mu_{n-1-d}$ is the
maximum number of edges in such a graph.  Note that the $n-1-d-k$ leaves with
eccentricity $d/2+1$ can be made adjacent to any of the vertices with
eccentricity $d/2$.

\medskip
{\bf Case 3:} \emph{$d=2$}.
Adding two edges incident to a vertex outside $P$ cannot give it eccentricity
$d/2$, since only dominating vertices have eccentricity $1$.  When $G$ has $k$
dominating vertices, with $k\ge1$, all other vertices have eccentricity $2$.
With $\CH k2+k(n-k)$ edges incident to the dominating vertices,
$\sigma_2(G)=\CH k2+2k(n-k)+4(m-\nu_k)$.  Having $k$ dominating vertices
requires at least $\CH k2+k(n-k)$ edges.  Since all edges cost $4$ in graphs
with diameter $2$ having no dominating vertices, $\sigma_2$ is minimized by
having dominating vertices, and this construction covers all numbers of edges
for which graphs with $n$ vertices and diameter $2$ exist.

\medskip
{\bf Case 4:} \emph{$d=4$ and $m\in\{n,n-2\}$}.
When $d=4$, there is a danger in the construction in Case 2.
If $\mu_{k-1}<m\le \mu_k$, then the construction in Case 2 achieving the
lower bound aims to create $k$ vertices with eccentricity $d/2$ outside $P$.
If $k<n-1-d$, then there is also a vertex $v$ made adjacent only to the central
vertex $z$ of $P$.  The vertices aimed to have eccentricity $d/2$ by being
adjacent to the two neighbors of $z$ on $P$ will have distance $3$ from $v$
unless they are made adjacent also to $z$.  That is, after making these
vertices adjacent to $z$ to complete a spanning tree, we still need $2k$
more edges, so we need $m\ge n-1+2k$.

Since $m>\mu_{k-1}$, we are given $m\ge n+\CH {k-1}2+2(k-1)$ when we want $k$
vertices with eccentricity $d/2$.  Thus it suffices to have
$\CH{k-1}2+2k-2\ge 2k-1$, which holds for $k\ge3$.  Indeed, since this
inequality fails only when $k\in\{1,2\}$ and then only by $1$, the construction
fails only when $k=1$ with $m=n$ and when $k=2$ with $m=n+2$.  Failure also
requires a vertex $v$ causing trouble, so $k<n-1-d$, which requires $n\ge 7$
when $k=1$ and $n\ge8$ when $k-2$.

We cannot meet the requirements for equality in the lower bound; for $k=1$ and
$m=n\ge7$, we cannot have a second vertex with eccentricity $2$.  All vertices
not on $P$ have eccentricity at least $3$.  In a spanning tree that grows to
include all the vertices, each edge costs at least $6$, with equality only if
all are incident to $z$, the center of $P$.  We must add one more edge, with
cost at least $9$.  Hence the total cost is at least $2\cdot 12+6(n-3)+9$,
which equals $6n+15$, raising the lower bound by $3$ compared to the general
formula.

When $m=n+2$ and $n\ge8$, we can still achieve the general lower bound!  It
begins with the path $P$ of length $4$ plus $n-7$ vertices adjacent to $z$, the
center of $P$, leaving $k=2$.  This gives $n-2$ vertices and $n-3$ edges.
Achieving the formula requires adding two vertices and five edges so that four
edges each cost $6$ and one costs $4$.  Add one vertex $u$ adjacent to the
three central vertices of $P$; it has eccentricity $2$.  Now add $w$ adjacent
to $u$ and $z$; it has eccentricity $3$, so its incident edges each cost only
$6$, as desired.  The total cost is $6n+22$.
\end{proof}

The detailed result allows us to minimize over all choices of $m$.
\begin{Corollary}
If $G$ is a connected graph with $n$ vertices and diameter $d$, then
$$
\sigma_2(G)\ge g_2(d)+(n-1-d)\CL{\FR{d}2}\CL{\FR{d}2+1}.
$$
Equality holds for the tree consisting of a path $P$ of length $d$ plus
$n-d-1$ leaves adjacent to the vertex (or either vertex) of $P$ with
eccentricity $\CL{d/2}$.
\end{Corollary}
\begin{proof}
Since the lower bound in Theorem~\ref{dmlow} increases as $m$ increases,
the minimum of $\sigma_2$ over $n$-vertex graphs with diameter $d$ is
achieved when the number of edges is minimized.  There does exist a tree with
diameter $d$, and we take one with smallest $\sigma_2$ by attaching vertices
outside a path of length $d$ in the cheapest way.  This is indeed the
instance of the construction in Theorem~\ref{dmlow} when $m=n-1$.
\end{proof}

\section{Bounds using chromatic number or clique number}

For various lower bounds, it is necessary to pay attention to the number
of dominating vertices.  The point is that dominating vertices have
eccentricity $1$, and when there is a dominating vertex the non-dominating
vertices have eccentricity $2$.

\begin{Observation} \label{ttt} \textnormal{\cite{DZT,TZ1}}
If $G$ is an $n$-vertex  connected graph having $s$ dominating vertices, where
$0\le s\le n$, then $\sigma_0(G)\ge2-\frac{s}{n}$ and
$\sigma_1(G)\ge4n-3s$.
In each bound, equality holds if and only if $\diam(G)\le2$.
\end{Observation}

For disjoint graphs $G$ and $H$, let $G+H$ denote their disjoint union,
and let $G\gjoin H$ be the \emph{join} of $G$ and $H$, obtained from $G+H$ by
adding as edges all pairs consisting of one vertex of $G$ and one vertex of $H$.
The disjoint union of $k$ copies of $G$ is denoted by $kG$.

We now consider extremal problems over $n$-vertex graphs with fixed chromatic
number or clique number.  When the chromatic number or clique number is $1$ or
$n$, the only instances are $\Kb_n$ and $K_n$, so the extremal problems are
trivial.  The lower bound for $\sigma_1$ in terms of clique number was
discussed in~\cite{DMCC}.

\begin{Theorem}\label{chrom01}
If $G$ is an $n$-vertex connected graph with chromatic number $k$, where
$2\le k\le n-1$, then $\sigma_0(G)\ge 2-\frac{k-1}{n}$ and
$\sigma_1(G)\ge 4n-3k+3$.  In each bound, equality holds if and only if
$G=K_{k-1}\gjoin \Kb_{n+1-k}$.
For clique number $k$ the bounds are the same,
with equality if and only if $G$ has $k-1$ dominating vertices.
\end{Theorem}

\begin{proof}
When partitioning $V(G)$ into $k$ color classes, dominating vertices must be in
classes of size $1$.  If $s$ is the number of dominating vertices in $G$, then
$s\le k-1$, since $k<n$.  By Observation~\ref{ttt}, the bounds on $\sigma_0$
and $\sigma_1$ follow.

Equality in the lower bound for $\sigma_0$ or $\sigma_1$ requires $s=k-1$.
Since $\chi(G)=k$, the remaining $n+1-k$ vertices must form another independent
set.  Hence $G= K_{k-1}\gjoin \Kb_{n+1-k}$.  Furthermore, equality holds for
this graph, by Observation~\ref{ttt}.

A graph with clique number $k$ has at most $k-1$ dominating vertices, so
again Observation~\ref{ttt} applies.
\end{proof}

The lower bound for $\sigma_2$ is more difficult; we postpone it and
consider upper bounds  on $\sigma_0$ and $\sigma_1$.
Recall the graph $B_{n,d}$ from Example~\ref{Bnd}; having diameter $d$,
it is forming by growing of a path of length $d-1$ from one vertex of a
complete graph with $n-d+1$ vertices.

In our next result, the values for $\sigma_0$ were obtained for given chromatic
number in \cite{DM} and for given clique number in \cite{DMCC}.
Our argument provides the bounds for both $\sigma_0$ and $\sigma_1$ via a
simpler proof avoiding complicated calculations.  Recall that $f_i(n,d)$
with $i\in\{0,1\}$ is the sharp upper bound on $\sigma_i$ for trees with
diameter $d$ from Lemma~\ref{treemax}.

\begin{Theorem}
Fix $i\in\{0,1\}$.  If $G$ is an $n$-vertex connected graph with chromatic
number $k$ or clique number $k$, where $2\le k\le n-1$,
then $\sigma_i(G)\le f_i(n,n-k+1)$, with equality when $G=B_{n,n-k+1}$.
\end{Theorem}

\begin{proof}
For $k=3$ and odd $n$, note that $\sigma_0(C_n)=(n-1)/2$ and
$\sigma_1(C_n)=n(n-1)^2/4$; these values are smaller than for $B_{n,n-2}$.
Hence we may assume that $G$ is not a complete graph and not an odd cycle.

Therefore, for chromatic number $k$, Brooks' Theorem implies
$\Delta(G)\ge \chi(G)=k$, where $\Delta(G)$ is the maximum degree of $G$.
The same is true for connected graphs with clique number $k$.  Hence $G$ has a
vertex of degree at least $k$ and thus a spanning tree $T$ with maximum degree
at least $k$.  When an $n$-vertex tree has maximum degree at least $k$, there
are at least $k-2$ vertices outside any path, and hence $\diam(T)\le n-k+1$.

Since $T$ is a subgraph of $G$, Observation~\ref{l0} yields
$\sigma_i(G)\le \sigma_i(T)$.
By Lemma~\ref{treemax}, since the bounds $f_i(n,d)$ strictly increase with
increasing $d$, we have $\sigma_i(T)\le f_i(n,n-k+1)$, with equality for
$T\in \bT^{n,n-k+1}$.  Finally, the vertices of $B_{n,n-k+1}$ have the same
eccentricities as the vertices in the member of $\bT^{n,n-k+1}$ that has
$k-1$ leaves with a common neighbor, so $\sigma_i(B_{n,n-k+1})=f_i(n,n-k+1)$,
and $B_{n,n-k+1}$ has chromatic number and clique number $k$.
\end{proof}

Next we consider $\sigma_2$, where we must pay close attention to the number
of edges.  We recall an elementary lemma.

\begin{Lemma}[{\rm \cite{EK}}] \label{edge}
Every $n$-vertex connected graph with chromatic number $k$ (or clique number
$k$) has at least $\CH k2+n-k$ edges.
\end{Lemma}
\begin{proof}
A $k$-critical subgraph with $p$ vertices must have at least $(k-1)p/2$ edges,
plus $n-p$ more to incorporate the other vertices.  For $k\ge3$, the number
of edges is minimized at $p=k$.

The argument holds more directly for graphs with clique number $k$.
\end{proof}

\begin{Theorem}\label{chrom2}
If $G$ is an $n$-vertex connected graph with chromatic number $k$ (or clique
number $k$), where $2\le k\le n-1$, then
$$
\sigma_2(G)\ge\begin{cases}
2n+2k^2-6k+2& \mbox{if $2\le k\le \FL{\FR{n+1}{2}}$} \\
(8k-3)n-2n^2-6k^2+4k-1&\mbox{if $\FL{\FR{n+1}{2}} <k< \CL{\FR{2n+1}3}$}\\
2(k-1)n-\frac{3k^2}{2}+\frac{5k}{2}-1&\mbox{if $\CL{\FR{2n+1}3}\le k\le n-1$.}
\end{cases}
$$
Also, equality holds in the bound if and only if
$G= K_s\gjoin(K_{k-s}+\Kb_{n-k})$, where $s=1$ in the first case,
$s$ is $4k-2n-2$ or $4k-2n-1$ in the second case, and $s=k-1$ in the third case.
\end{Theorem}

\begin{proof}
It suffices to prove the lower bound for $n$-vertex connected graphs with
chromatic number $k$ and achieve it by a $n$-vertex connected graph with
clique number $k$.  For the lower bound, let $S$ be the set of dominating
vertices in $G$, and let $s=|S|$.

First suppose $s\ge1$, so $s$ vertices have eccentricity $1$ and the others
have eccentricity $2$.  The value of $\vareps_G(u)\vareps_G(v)$ is $1$ when
$u,v\in S$, is $2$ when exactly one of $u,v$ is in $S$, and is $4$ when
$u,v\notin S$.  Also $G= K_s\gjoin H$ with $\chi(H)=k-s$, but $H$ need not be
connected, so we can only guarantee edges joining color classes, yielding
$|E(H)|\ge\CH{k-s}2$.  Thus
\begin{eqnarray*}
\sigma_2(G)&\ge&\CH s2+2 s(n-s)+4\CH{k-s}2\\
&=&\CH s2+2s(n-2k+1)+2k(k-1)\\
&=&h(s)+2k(k-1),
\end{eqnarray*}
where $h(s)=\CH s2+2s(n-2k+1)$.

For $G$ with fixed $n$ and $k$, we may have various $s$.  We obtain a lower
bound by choosing $s$ to minimize $h(s)$, which may occur at different values
depending on the relationship between $n$ and $k$.  Achieving equality in the
bound then requires $G= K_s\gjoin(K_{k-s}+\Kb_{n-k})$.

Note that $h(s)$ is quadratic in $s$, minimized at $s=4k-2n-\frac{3}{2}$.  Thus
$4k-2n-1$ and $4k-2n-2$ are candidates for the minimizing value of $s$ (and $h$
has the same value at the two points), but also we require $1\le s\le k-1$.

If $k\le\FL{\FR{n+1}2}$, then $4k-2n-1\le 1$, so $h(s)$ is minimized only at
$s=1$.  If $\FL{\FR{n+1}{2}}<k<\CL{\FR{2n+1}{3}}$, then $4k-2n-1>1$ and
$4k-2n-2<k-1$, so $h(s)$ is minimized at $4k-2n-2$ and $4k-2n-1$.  If
$k\ge \CL{\FR{2n+1}{3}}$, then $4k-2n-2\ge k-1$, so $h(s)$ is minimized only
at $s=k-1$.

It remains only to show that for graphs with no dominating vertices, the value
of $\sigma_2(G)$ with given $n$ and $k$ is always larger than the minimum
of $\sigma_2(G)$ for $s\ge 1$.  When $s=0$, each vertex in $G$ has eccentricity
at least $2$, so $\vareps_G(u)\vareps_G(v)\ge4$ for every edge $uv\in E(G)$.
By Lemma~\ref{edge}, $|E(G)|\ge \CH k2+n-k$.  Thus
\begin{eqnarray*}
\sigma_2(G)&\ge& 4\left[ \CH k2+n-k\right]
~=~4n+2k^2-6k\\
&>&2n+2+2k^2-6k ~=~ h(1)+2k(k-1) .
\end{eqnarray*}
Since the value of $h(1)$ is always at least the smallest value of $h(s)$,  the
value of $\sigma_2(G)$ when $s=0$ is always larger than the minimum value of
$\sigma_2(G)$ for $s\ge 1$, as desired.
\end{proof}

For large values of $\sigma_2(G)$ when $k$ is fixed, we present a construction.
We first stratify by the diameter.

\begin{Example}\label{kdhigh}
As a building block, we use the \emph{Tur\'an graph} $T_{n,r}$, which
generalizes the complete graph $K_r$.  It consist of $n$ vertices grouped into
independent sets of size $\FL{n/r}$ or $\CL{n/r}$, with vertices adjacent when
they belong to different such sets.  As proved originally by Tur\'an~\cite{T},
it is well known that $T_{n,r}$ has the most edges among $n$-vertex graphs with
chromatic number $r$ (or clique number $r$).

By arranging for the vertices in such a subgraph to have large eccentricity, we
can construct a graph with given chromatic number and clique number for which
$\sigma_2$ is large.  Let $U_{n,k,d}$ be the graph obtained from
$T_{n-d+1,k}+P_{d-1}$ by making one endpoint $z$ of the copy of $P_{d-1}$
adjacent to all the vertices in a largest independent set in the Tur\'an graph.
The vertex $z$ has eccentricity $d-2$ (when $d\ge4$), its neighbors in the copy
of $T_{n-d+1,k}$ have eccentricity $d-1$, and the other vertices in the copy of
$T_{n-d+1,k}$ have eccentricity $d$.  The diameter is $d$.  The graph $B_{n,d}$
discussed earlier is $U_{n,n-d+1,d}$.

To compute $\sigma_2(U_{n,k,d})$ exactly when $d\ge4$, let $n'=n-d+1$.
In order to count the contributions from $T_{n',k}$ more simply, we subtract
from $g_2(d)$ the contributions of the last two edges on that end of a
diametric path, which are $d(d-1)$ and $(d-1)(d-2)$.  We then count $d^2$
for each edge of the copy of $T_{n',k}$, subtracting $d$ for each edge that
actually contributes only $d(d-1)$.  Altogether,
$$
\sigma_2(U_{n,k,d})=g_2(d)+t_{n',k}d^2-2(d-1)^2
-d\CL{\FR{n'}k}\left(n'-\CL{\FR{n'}k}\right)+\CL{\FR{n'}k}(d-1)(d-2) .
$$
The number of edges in $T_{n',k}$ is approximately $(1-1/k)n'^2/2$, with
equality when $k$ divides $n'$.  For fixed $k$, with $d$ growing as a constant
fraction of $n$, asymptotically,
$$
\sigma_2(U_{n,k,d})= \FR12(1-1/k)(n-d+1)^2d^2+O(d^3+d(n-d)).
$$
For fixed $k$ this value is maximized by setting $d=(n+1)/2$,
and we obtain a graph $G$ with
$\sigma_2(G)=\FR 1{32}(1-1/k)n^4+O(n^3)$.
\end{Example}

Recall from Corollary~\ref{s2alln} that over all $n$-vertex graphs the
maximum of $\sigma_2(G)$ is $n^4/32+O(n^3)$.
Example~\ref{kdhigh} shows that we can reach asymptotically half of this
bound already using bipartite graphs.  If $k$ grows as an unbounded function of
$n$, then for any $\epsilon$ the ratio of $\sigma_2(G)$ to the upper bound
exceeds $1-\epsilon$ for sufficiently large $n$.  Nevertheless, it seems
likely that, for fixed $k$, the construction $U_{n,k,n/2}$ is asymptotically
optimal, meaning that the upper bound needs improvement.

\section{Bounds in terms of matching number}

Recall that $\alpha'(G)$ denotes the matching number, the maximum size of a
matching in $G$.  Always $\alpha'(G)\ge1$ when $G$ is a connected graph with at
least two vertices, with equality if and only if $G$ is a star or the triangle
$K_3$, so we restrict attention to $\alpha'(G)\ge2$.  Always
$\alpha'(G)\le\FL{n/2}$, and when equality holds it is possible
for all vertices to be dominating vertices, in which case $G= K_n$.

\begin{Lemma}\label{tm}
If $G$ is an $n$-vertex connected graph with matching number $k$, where
$2\le k<\FL{n/2}$, then the number of dominating vertices in $G$ is at most
$k$.
\end{Lemma}
\begin{proof}
Let $s$ be the number of dominating vertices in $G$.  When $k<\FL{n/2}$, a
maximum matching $M$ leaves at least two vertices uncovered, and hence no
uncovered vertex can be a dominating vertex.  If both endpoints of an edge in
$M$ are dominating vertices, then they can instead be matched to uncovered
vertices to obtain a larger matching.  Hence at most one vertex from each edge
of $M$ (and no other vertices) can be a dominating vertex.
\end{proof}

Note that $\sigma_0\ge1$ and $\sigma_1\ge n$ for every $n$-vertex connected
graph, with equality only for $K_n$.  Hence for the study of lower bounds
on $\sigma_0$ and $\sigma_1$ in terms of the matching number,  we may restrict
our attention to graphs $G$ with $\alpha'(G)<\FL{n/2}$.

\begin{Proposition} \label{matching1}
If $G$ is an $n$-vertex connected graph with matching number $k$, where
$2\le k <\FL{n/2}$, then $\sigma_0(G)\ge 2-\frac{k}{n}$ and
$\sigma_1(G)\ge 4n-3k$.  In each bound, equality holds if and only if
$G= K_k \gjoin \Kb_{n-k}$.
\end{Proposition}
\begin{proof}
Let $s$ be the number of dominating vertices in $G$.  By Observation~\ref{ttt}
and Lemma~\ref{tm}, we have $\sigma_0(G)\ge 2-\FR sn\ge 2-\FR kn$ and
$\sigma_1(G)\ge 4n-3s\ge 4n-3k$, since $2\le k<\FL{n/2}$.
To achieve equality in either bound, we must have exactly $k$ dominating
vertices, so $G=K_k\gjoin H$.  Any edge in $H$ permits a larger matching,
since $k<\FL{n/2}$ implies $n-k\ge k+2$.  Hence equality requires
$G= K_k \gjoin \Kb_{n-k}$.
\end{proof}

Now we consider the upper bounds on $\sigma_0$ and $\sigma_1$ in terms of
matching number.  Again $f_i(n,d)$ denotes the value of $\sigma_i(T)$
for $T\in \bT^{n,d}$, the maximum over $n$-vertex trees with diameter $d$,
as stated in Lemma~\ref{treemax}.

\begin{Lemma} \label{l3} \textnormal{\cite{DZT, TZ1}}
Fix $i\in\{0,1\}$.  Let $T$ be an $n$-vertex tree with matching number $k$.
If $2\le k<\FL{n/2}$, then $\sigma_i(T)\le f_i(n,2k)$, with equality if and
only if $T\in\bT^{n,2k}$. If $k=\FL{n/2}$, then $\sigma_i(T)\le \sigma_i(P_n)$,
with equality if and only if $T=P_n$.
\end{Lemma}

\begin{Theorem}
Fix $i\in\{0,1\}$.  Let $G$ be an $n$-vertex connected graph with  matching
number $k$.  If $2\le k<\FL{n/2}$, then $\sigma_i(G)\le f_i(n,2k)$, with
equality for $G\in\bT^{n,2k}$.  If $k=\FL{n/2}$, then
$\sigma_i(G)\le \sigma_i(P_n)$, with equality if and only if $G= P_n$.
\end{Theorem}

\begin{proof}
Every connected graph $G$ has a spanning tree with the same matching number,
because a maximum matching in $G$ can be grown to a spanning tree by
iteratively adding an edge of $G$ joining two components of the current
subgraph until $n-1$ edges are obtained.  The resulting tree $T$ cannot have a
larger matching, since it would also be a matching in $G$.

Invoking Observation~\ref{l0} and Lemma~\ref{l3}, we thus have
$\sigma_i(G)\le \sigma_i(T)\le f_i(n,2k)$ if $2\le k<\FL{n/2}$ and
$\sigma_i(G)\le \sigma_i(T)\le \sigma_i(P_n)$ if $k<\FL{n/2}$.
\end{proof}

The restriction to $\alpha'(G)<\FL{n/2}$ no longer applies when we discuss
$\sigma_2$.  In particular, we include $k=\FL{n/2}$ in the following theorem.

\begin{Theorem}
Let $G$ be an $n$-vertex connected graph with matching number $k$, where
${2\le k\le\FL{n/2}}$.  The case $4\le n\le 6$ and $k=\FL{n/2}$ is exceptional,
with $\sigma_2(G)\ge \CH n2$ and equality if and only if $G= K_n$.  Otherwise,
$\sigma_2(G)\ge 2n+4k-6$, with equality if and only if
$G= K_1\gjoin \left((k-1)K_2+\Kb_{n+1-2k}\right)$.
\end{Theorem}

\begin{proof}
First consider the case $(n,k)=(4,2)$.  Checking individual graphs shows
$\sigma_2(G)\ge\CH42$, with equality if and only if $G= K_4$.  Henceforth
we exclude this exceptional case.

The $n$-vertex connected graph $K_1\gjoin \left((k-1)K_2+\Kb_{n+1-2k}\right)$
has matching number $k$.  It has one dominating vertex (eccentricity $1$)
and $n-1$ vertices with  eccentricity $2$, so the value of $\sigma_2$ on this
graph is $2n+4k-6$.

Let $s$ be the number of dominating vertices in $G$.  If $s=0$, then every
vertex has eccentricity at least $2$, so $\sigma_2(G)\ge 4|E(G)|\ge 4(n-1)$,
since $G$ is connected.  Since $k\le {n}/{2}$, we have $2n+4k-6<4(n-1)$.  Thus
$\sigma_2\left(K_1\gjoin \left[(k-1)K_2+\Kb_{n+1-2k}\right]\right)< 4(n-1)$.
Hence over this class $\sigma_2$ cannot be minimized by a graph with no
dominating vertex.

Thus we may assume $s\ge1$ and $G=K_s\gjoin H$, where $H$ has $n-s$ vertices
and maximum degree less than $n-s-1$.  Since dominating vertices have
eccentricity $1$ and those in $H$ have eccentricity $2$, we have
$\sigma_2(G)=\CH s2+2s(n-s)+4|E(H)|$.  We consider two ranges for $s$.

If $1\le s\le k$, then $\alpha'(G)=k$ requires $|E(H)|\ge k-s$.
Note that $2n-\FR32s-6>0$, since $3s<4n-12$ when $s\le k\le\FL{n/2}$ and
$(n,k)\neq(4,2)$.  Hence
\begin{eqnarray*}\label{mats2}
\sigma_2(G)&\ge & \CH s2+2s(n-s)+4(k-s)\qquad\qquad\qquad(4.1)\\
&=&2n+4k-6+(s-1)\left(2n-\FR32s-6\right)\\
&\ge&2n+4k-6.\qquad\qquad\qquad\qquad\qquad\qquad(4.2)
\end{eqnarray*}
Equality holds in $(4.1)$ if and only if
$G=K_s\gjoin\left((k-s)K_2+\Kb_{n+s-2k}\right)$.
In $(4.2)$, equality holds if and only if $s=1$, since $3s<4n-12$.
Thus $\sigma_2(G)\ge2n+4k-6$, with equality if and only if
$G= K_1\gjoin \left((k-1)K_2+\Kb_{n+1-2k}\right)$.

By Lemma~\ref{tm}, the case $s>k$ occurs only when $k=\FL{n/2}$.
We have noted that $G= K_s\gjoin H$ and $\sigma_2(G)=\CH s2+2s(n-s)+4|E(H)|$.
Since $s>\FL{n/2}$, we have a matching of size $k$ that matches vertices of
$H$ to dominating vertices regardless of any edges in $H$, so $\sigma_2(G)$ is
minimized (only) by setting $H=\Kb_{n-s}$.  Let
$G_s=K_s\gjoin \Kb_{n-s}$, so
$\sigma_2(G_s)=\CH s2+2s(n-s)=-\frac{3}{2}s^2+(2n-\frac{1}{2})s$. Note that
$G_n=K_n$.  Over $\lfloor\frac{n}{2}\rfloor+1\le s\le n$, we have
\begin{eqnarray*}
\min \sigma_2(G_s)&=&\min\{\sigma_2(G_{\lfloor\frac{n}{2}\rfloor+1}),\sigma_2(K_n)\}\\
                &=&\begin{cases}
\min\{\frac{5}{8}n^2-\frac{5}{8},\CH n2\} & \mbox{if $n$ is odd}\\
\min\{\frac{5}{8}n^2+\frac{1}{4}n-2,\CH n2\} & \mbox{if $n$ is even}
                   \end{cases} \\
                   &=& \CH n2.
\end{eqnarray*}
Thus over $s>k$, the value of $\sigma_2$ is minimized by choosing $G= K_n$,
where it equals $\CH n2$.

Comparing these two ranges for $s$, we have
$\sigma_2(G)\ge \min\{2n+4k-6,\CH n2\}$.  Since $2\le k\le\FL{n/2}$, direct
calculation yields $2n+4k-6>\CH n2$ if $4\le n\le 6$ and $k=\FL{n/2}$,
while $2n+4k-6<\CH n2$ if $n\ge7$ or if $4\le n\le 6$ and $k<\FL{n/2}$.

Hence the minimum $\sigma_2$ is achieved uniquely by
$K_1\gjoin \left((k-1)K_2+\Kb_{n+1-2k}\right)$ unless
$4\le n\le 6$ and $k=\FL{n/2}$.  In this exceptional case,
the minimum $\sigma_2$ is $\CH n2$, achieved uniquely
by $K_n$ (including the case $(n,k)=(4,2)$).
\end{proof}

Finally, we offer a construction for large $\sigma_2(G)$ when $\alpha'(G)=k$.

\begin{Example}
We fix $k$ and $n$ and stratify by the diameter, $d$.  Along any path with
length $d$ there is a matching with size $\CL{d/2}$, so we must have $d\le 2k$.
In fact, we take $d<2k$.

When $\alpha'(G)=k$, there are $n-2k$ vertices outside a maximum matching.  To
ensure this while having many edges with cost $d^2$, start with a path having
$d-1$ vertices, append $n-2k+1$ leaves at one end, and append $2k-d$ leaves at
the other end.  Add edges to make these $2k-d$ vertices into a clique.
Any matching omits $n-2k$ vertices among the leaves, but the rest can be
covered (assuming that $n$ is even).  Let $C_{n,k,d}$ be the resulting graph.
We have $\sigma_2(C_{n,k,d})=g_2(d)+(n-d-1)d(d-1)+\CH{2k-d}2 d^2$.
Setting $d=k$ yields a construction where $\sigma_2$ is
$k^4/2+k^2(n-k)+k^3/12+O(k^2)$ when $k$ grows as a constant fraction of $n$.
When $k$ is quite small, the dominant term changes, and then it is better to
take $d=2k-1$.

When $k=\FL{n/2}$, this construction reduces to $B_{n,d}$, which we expect to
generally be extremal without the restriction on matching number.  On the
other hand, when the matching number is only $\CL{d/2}$, we are forced back
toward the realm of trees.  We leave the resolution of the upper bound on
$\sigma_2$ in these classes for future research.
\end{Example}


\begin{thebibliography}{99}
\frenchspacing

\bibitem{BH} F. Buckley, F. Harary, {\it Distance in Graphs\/},
Addison-Wesley, Redwood City, Carlifornia, 1990.

\bibitem{DGS} P. Dankelmann, W. Goddard, C.S. Swart,
{\it The average eccentricity of a graph and its subgraphs\/},
Util. Math. 65 (2004) 41--51.

\bibitem{DM} P. Dankelmann, S. Mukwembi,
{\it Upper bounds on the average eccentricity\/},
Discrete Appl. Math. 167 (2014) 72--79.

\bibitem{DO} P. Dankelmann, F.J. Osaye,
{\it Average eccentricity, minimum degree and maximum degree in graphs\/}
J. Comb. Optim. 40 (2020) 697--712.

\bibitem{DOMR} P. Dankelmann, F.J. Osaye, S. Mukwembi, B.G. Rodrigues,
{\it Upper bounds on the average eccentricity of $K_3$-free and $C_4$-free
graphs\/}, Discrete Appl. Math.  270 (2019) 106--114.

\bibitem{DLG} K.C. Das, D.W. Lee, A. Graovac,
{\it Some properties of Zagreb eccentricity indices\/},
Ars Math. Contemp. 6 (2013) 117--125.

\bibitem{DMCC} K.C. Das, A.D. Maden, I.N. Cang{\"u}l, A.S. \c{C}evik,
 {\it On average eccentricity of graphs\/},
Proc. Nat. Acad. Sci. India, Sect. A  87 (2017) 23--30.

\bibitem{DI} Z. Du, A. Ili{\'c},
{\it On AGX conjectures regarding average eccentricity\/},
MATCH Commun. Math. Comput. Chem. 69 (2013) 597--609.

\bibitem{DZT} Z. Du, B. Zhou, N. Trinajsti{\'c},
{\it  Extremal properties of the Zagreb eccentricity indices\/},
Croat. Chem. Acta 85 (2012) 359--362.

\bibitem{EK} A.P. Er\v{s}ov, G.I. Ko\v{z}uhin,
{\it Estimates of the chromatic number of connected graphs\/},
(Russian) Dokl. Akad. Nauk SSSR 142 (1962) 270--273.

\bibitem{GH} M. Ghorbani, M.A. Hosseinzadeh,
{\it A new version of Zagreb indices\/},
Filomat  26 (2012) 93--100.

\bibitem{GRTW} I. Gutman, B. Ru\v s\v ci\'c, N. Trinajsti\'{c}, C.F. Wilcox,
{\it Graph theory and molecular orbitals. XII. Acyclic polyenes\/},
 J. Chem. Phys. 62 (1975) 3399--3405.

\bibitem{GT} I. Gutman, N. Trinajsti\'c,
{\it Graph theory and molecular orbitals.  Total $\varphi$-electron energy of
alternant hydrocarbons\/},
Chem. Phys. Lett. 17 (1972) 535--538.

\bibitem{H} F. Hayat,
{\it The minimum second Zagreb eccentricity index of graphs with parameters\/},
Discrete Appl. Math. 285 (2020) 307--316.

\bibitem{HHMRD} P. Hauweelea, A. Hertz, H. M\'elot, B. Ries, G. Devillez,
{\it Maximum eccentric connectivity index for graphs with given diameter\/},
Discrete Appl. Math. 268 (2019) 102--111.

\bibitem{HLT} C. He, S. Li, J. Tu,
{\it Edge-grafting transformations on the average eccentricity of graphs and
their applications\/},
Discrete Appl. Math. 238 (2018) 95--105.

\bibitem{HBDAA} B. Horoldagva, L. Buyantogtokh, S. Dorjsembe, E. Azjargal,
D. Adiyanyam, {\it On graphs with maximum average  eccentricity\/},
Discrete Appl. Math. 301 (2021) 109--117.

\bibitem{I} A. Ili{\'c},
{\it On the extremal properties of the average eccentricity\/},
Comput.  Math. Appl. 64 (2012) 2877--2885.

\bibitem{KS} E.V. Konstantinova, V.A. Skorobogatov,
{\it Molecular hypergraphs: The new representation of nonclassical molecular
structures with polycentric delocalized bonds\/},
J. Chem. Inf. Comput. Sci. 35 (1995) 472--478.


\bibitem{LZ} J. Li, J. Zhang,
{\it On the second Zagreb eccentricity indices of graphs\/},
Appl. Math. Comput. 352  (2019) 180--187.



\bibitem{O} O. Ore,
{\it Diameters in graphs\/},
J. Combin. Theory 5 (1968) 75--81.

\bibitem{QD} X. Qi, Z. Du,
{\it On Zagreb eccentricity indices of trees\/},
MATCH Commun. Math.  Comput. Chem. 78 (2017) 241--256.

\bibitem{QZL} X. Qi, B. Zhou, J. Li,
{\it Zagreb eccentricity indices of unicyclic graphs\/},
Discrete Appl. Math. 233 (2017) 166--174.

\bibitem{QZ} P. Qiao, X. Zhan,
{\it The largest graphs with given order and diameter: a simple proof\/},
Graphs Combin. 35 (2019) 1715--1716.

\bibitem{SLH} X. Song, J. Li, W. He,
{\it On Zagreb eccentricity indices of cacti\/},
Appl. Math. Comput. 383  (2020) 125361.

\bibitem{TQ}  Y. Tang, X. Qi,
{\it Ordering graphs with large eccentricity-based topological indices\/}
J. Inequal. Applic. (2021) 2021:24, 11 pages.

\bibitem{TZ1} Y. Tang, B. Zhou,
{\it On average eccentricity\/},
MATCH Commun. Math. Comput. Chem. 67 (2012) 405--423.

\bibitem{TZ2} Y. Tang, B. Zhou,
{\it Ordering unicyclic graphs with large average eccentricities\/},
Filomat 28:1 (2014) 207--210.

\bibitem{TC2}  R. Todeschini, V. Consonni,
Molecular descriptors for chemoinformatics,  Wiley-VCH, Weinheim (2009).

\bibitem{T} P. Tur\'an,
{\it Eine Extremalaufgabe aus der Graphentheorie\/} (Hungarian)
Mat. Fiz. Lapok 48 (1941), 436--452.

\bibitem{VG}D. Vuki{\v{c}}evi{\'c}, A. Graovac,
{\it Note on the comparison of the first and second normalized Zagreb
eccentricity indices\/}, Acta Chem. Slov. 57 (2010) 524--538.


\bibitem{XZT} R. Xing, B. Zhou, N. Trinajsti{\'c},
{\it On Zagreb eccentricity indices\/},
Croat. Chem. Acta 84 (2011) 493--497.

\end{thebibliography}
\end{document}